\title{Extremal Hypercuts and Shadows of Simplicial Complexes }

\author{Nati Linial\thanks{School of Computer Science and engineering, The Hebrew University of Jerusalem, Jerusalem, Israel. Email: {\tt nati@cs.huji.ac.il}. Research supported by an ERC grant "High-dimensional combinatorics".}
\and Ilan Newman\thanks{
Department of Computer Science, University of Haifa, Haifa, Israel. Email: {\tt ilan@cs.haifa.ac.il}. This Research was supported by The Israel Science Foundation (grant number 862/10.)}
\and Yuval Peled\thanks{School of Computer Science and engineering, The Hebrew University of Jerusalem, Jerusalem, Israel. Email: {\tt yuvalp@cs.huji.ac.il}. Yuval Peled is grateful to the Azrieli Foundation for the award of an Azrieli Fellowship.}
\and Yuri Rabinovich\thanks{
Department of Computer Science, University of Haifa, Haifa, Israel. Email: {\tt yuri@cs.haifa.ac.il}. This Research was supported by The Israel Science Foundation (grant number 862/10.)}
}

\date{}

\date{\today}
\documentclass[11pt]{article}
\usepackage{amssymb,fullpage}
\usepackage{xspace}
\usepackage{latexsym}
\usepackage{times}
\usepackage{amsfonts}
\usepackage{amsmath}

\usepackage{mathrsfs}
\usepackage{bbm}

\usepackage{verbatim}

\usepackage{amsthm}
\usepackage{amsmath}
\usepackage{amssymb}

\usepackage{graphicx}
\usepackage{epstopdf}
\usepackage{titling}

\usepackage{enumerate}

\usepackage{algpseudocode}
\usepackage{algorithm}
\usepackage{algorithmicx}
\usepackage{tikz}

\newcommand{\ignore}[1]{}

\newcommand{\Z}{\ensuremath{\mathbb Z}}
\newcommand{\Q}{\ensuremath{\mathbb Q}}
\newcommand{\N}{\ensuremath{\mathbb N}}

\newcommand{\f}{\ensuremath{\mathbb F}}
\newcommand{\floor}[1]{\ensuremath{\left \lfloor{#1}\right \rfloor }}

\newtheorem{theorem}{Theorem}[section]
\newtheorem{proposition}[theorem]{Proposition}
\newtheorem{lemma}[theorem]{Lemma}
\newtheorem{claim}[theorem]{Claim}
\newtheorem{corollary}[theorem]{Corollary}
\newtheorem{observation}[theorem]{Observation}
\newtheorem{fact}[theorem]{Fact}

\newtheorem{conjecture}[theorem]{Conjecture}
\newtheorem{definition}[theorem]{Definition}

\newtheorem{remark}[theorem]{Remark}

\def \link {{\rm link}}
\def \sign {{\rm sign}}
\def \dim {{\rm dim}}
\def \im {{\rm Im}}
\def \ker {{\rm ker}}
\def \rank {{\rm rank}}

\def \Kn {K_n^{(2)}}

\begin{document}
\maketitle
\pagenumbering {arabic} 

\def \e {\epsilon}
\def \d {\delta}
\def \P {\mathcal P}

\def \noproof
 {\hspace*{0pt} \hfill {\quad \vrule height 1ex width 1ex depth 0pt}}
\def \xor {\oplus}

\begin{abstract}
Let $F$ be an $n$-vertex forest.
We say that an edge $e\notin F$ is in the {\em shadow} of $F$ if
$F\cup\{e\}$ contains a cycle. It is easy to see that if $F$ is ``almost a tree", that is, it
has $n-2$ edges, then at least $\lfloor\frac{n^2}{4}\rfloor$ edges are in its
shadow and this is tight. Equivalently, the largest number
of edges an $n$-vertex {\em cut} can
have is $\lfloor\frac{n^2}{4}\rfloor$. These notions have natural analogs in higher $d$-dimensional simplicial complexes, graphs being the case $d=1$. The results in dimension $d>1$ turn out to be remarkably different from the case in graphs. In particular the corresponding bounds
depend on the underlying field of coefficients. We find the (tight) analogous
theorems for $d=2$. We construct $2$-dimensional ``$\Q$-almost-hypertrees" (defined below) with an {\em empty} shadow. We also show that the shadow of an ``$\f_2$-almost-hypertree" cannot be empty, and its least possible density is $\Theta(\frac{1}{n})$. In addition we construct very large hyperforests with a shadow that is empty over every field.

For $d\ge 4$ even, we construct $d$-dimensional $\f_2$-almost-hypertree whose shadow has density $o_n(1)$.

Finally, we mention several intriguing open questions.

\end{abstract}

\section{Introduction}

This article is part of an ongoing research effort to bridge
between graph theory and topology (see, e.g.~\cite{kalai, sum_complex,DNRR,lin_mesh,bab_kah,farber,gromov,lubotzky}).
This research program starts from the observation that a graph can be viewed as a $1$-dimensional simplicial complex, and that many basic concepts of graph theory such as connectivity, forests, cuts, cycles, etc., have natural counterparts in the realm of higher-dimensional simplicial complexes.
As may be expected, higher dimensional objects tend to be more complicated
than their $1$-dimensional counterparts, and many fascinating phenomena reveal themselves from the present vantage point. This paper is dedicated to
the study of several extremal problems in this domain. We start by introducing some of the
necessary basic notions and definitions.

\vspace{0.7cm}
\noindent
{\bf Simplices, Complexes, and the Boundary Operator:} \\
%

All simplicial complexes considered here have $[n]=\{1,\ldots,n\}$ or $\Z_n$ as their {\em vertex set} $V$. A simplicial complex $X$ is a collection of subsets of $V$ that is closed under taking subsets. Namely, if $A\in X$ and $B\subseteq A$, then $B\in X$ as well. Members of $X$ are called {\em faces} or {\em simplices}. The {\em dimension} of the simplex $A\in X$ is defined as $|A|-1$.
A $d$-dimensional simplex is also called a $d$-simplex or a $d$-face for short. The dimension $\dim(X)$ is defined as $\max\dim(A)$ over all faces $A\in X$, and we also refer to a $d$-dimensional simplicial complex as a $d$-complex. The {\em size} $|X|$ of a $d$-complex $X$ is the number $d$-faces in $X$.
The complete $d$-dimensional complex $K_n^d = \{\sigma \subset [n] ~|~
|\sigma| \leq d+1 \}$ contains all simplices of dimension $\leq d$. If $X$ is a $d$-complex and $t < d$, the collection of all faces of dimension $\le t$ in $X$ is a simplicial complex that we call the $t$-{\em skeleton} of $X$. If this $t$-skeleton coincides with $K_n^t$ we say that $X$ has a {\em full} $t$-dimensional skeleton. If $X$ has a full $(d-1)$-dimensional skeleton (as we usually assume), then its {\em complement}  $\bar X$ is defined by taking a full $(d-1)$-dimensional skeleton and those $d$-faces that are not in $X$. 

The permutations on the vertices of a face $\sigma$ are split in two
{\em orientations} of $\sigma$, according to the permutation's
sign. The {\em boundary operator} $\partial=\partial_d$ maps an oriented
$d$-simplex $\sigma = (v_0,...,v_d)$ to the formal sum
$\sum_{i=0}^{d}(-1)^i(\sigma\setminus v_i)$, where $\sigma\setminus
v_i=(v_0,...v_{i-1},v_{i+1},...,v_d)$ is an oriented
$(d-1)$-simplex. We fix some field $\f$ and linearly extend
the boundary operator to free $\f$-sums of
simplices. We consider the ${n \choose d}\times{n \choose d+1}$
matrix form of $\partial_d$ by choosing arbitrary orientations for
$(d-1)$-simplices and $d$-simplices in $K_n^d$. Note that changing the
orientation of a $d$-simplex (resp. $d-1 $-simplex) results in
multiplying the corresponding column (resp. row) by $-1$. Thus the
$d$-boundary of a weighted sum of $d$ simplices, viewed as a vector $z$
(of weights) of dimension ${n \choose d+1}$ is just the matrix-vector
product $\partial_d z$. We denote by $M_X$ the submatrix of $\partial_d$ restricted to the columns associated with $d$-faces of a $d$-complex $X$.

The specific underlying fields that we consider in this paper are $\Q$ or $\f_2$. (In the latter case orientation is redundant). It is very interesting to extend the discussion to the case where everything is done over a commutative ring, and especially over $\Z$, but we do not do this here. We associate each column in the matrix form of $\partial_d$
with the corresponding $d$-simplex $\sigma$. This is an ${n \choose d}$-dimensional vector of $0,\pm 1$ whose support corresponds to the boundary of $\sigma$.
It is standard and not hard to see that for every choice of ground field, the matrix
$\partial_d$ has rank ${n-1 \choose d}$. A fundamental (easy) fact is that
$\partial_{d-1} \cdot \partial_{d} = 0$ for any $d$.

\vspace{0.7cm}
\noindent
{\bf  Rank Function and other notions:} \\
If $S$ is a collection of $n$-vertex $d$-simplices, then we define\footnote{In the language of simplicial homology, consider the $d$-complex $K(S)$ whose set of $d$-faces is $S$. Then, $\rank(S)$ is
just the dimension of $B_{d-1}$, the linear space of $(d-1)$-boundaries of $K(S)$,
and $|S| - \rank(S)$ is the dimension of the homology group $H_d(K(S))$.}
its {\em rank} as the $\f$-rank of the set of the corresponding columns of $\partial_d$.
(It clearly does not depend on the choice of orientations). The set of all $d$-faces of $K_n^d$ has rank ${{n-1}\choose d}$,
with basis being, e.g., the collection of all $d$-simplices that contain the
vertex $1$.

If $\rank(S)=|S|$, we say that $S$ is {\em acyclic} over $\f$.
A maximal acyclic set of $d$-faces is called a {\em $d$-hypertree},
and an acyclic set of size ${{n-1}\choose d}-1$ is called an {\em almost-hypertree}.
Hypertrees over $\Q$ were studied e.g., by Kalai~\cite{kalai} and others~\cite{adin,duval}
in the search for high-dimensional analogs of Cayley's formula for the number of labeled trees. A {\em $d$-dimensional hypercut} (or {\em $d$-hypercut} in short) is an inclusion-minimal set of $d$-faces that intersects every hypertree. It is a standard fact in matroid theory that for every hypercut $C$, there is a hypertree $T$ such that $|C \cap T|=1$.

The {\em shadow} $SH(S)$ of a set $S$ of $d$-simplices consists of all $d$-simplices $\sigma\not\in S$ which are in the $\f$-linear span of $S$, i.e., such that $\rank (S\cup\sigma) = \rank(S)$. A set of $d$-faces is a hypercut iff its complement is the union of an almost-hypertree and its shadow. If $SH(S) = \emptyset$, we say that $S$ is {\em closed} or {\em shadowless}. For instance, a set of edges in a graph is closed if it is a disjoint union of cliques.

We turn to define $d$-{\em collapsibility}. A $(d-1)$-face $\tau$ in a
$d$-complex $K$ is called {\em exposed} if it is contained in exactly one
$d$-face $\sigma$ of $K$. An elementary $d$-collapse on $\tau$ consists of
the removal of $\tau$ and $\sigma$ from $K$. We say that $K$ is $d$-collapsible
if it is possible
to eliminate all the $d$-faces of $K$ by a series of elementary $d$-collapses.
It is an easy observation that the set of $d$-faces in a $d$-collapsible $d$-complex
is acyclic over every field.

We refer the reader to Section~\ref{section:combin} for some more background material.

%

\vspace{1.6cm}
\noindent
{\bf Results:}

In this paper we study extremal problems concerning the possible sizes
of hypercuts and shadows in simplicial complexes.

We begin with some trivial observations on $n$-vertex graphs, starting with the (non-tight) claim that no cut can have more than ${n\choose 2} - n+2$ edges. This follows, since for every cut there is a tree that meets it in exactly one edge. Actually the largest number of edges of a cut is $\lfloor \frac{n^2}{4} \rfloor$. We investigate here the $2$-dimensional situation and discover that it is completely different from the graphical case. When we discuss $2$-complexes, we refer to $2$-faces as faces (and keep the terms vertex and edge for $0$ and $1$ dimensional faces).

A $2$-dimensional hypertree has ${n-1 \choose 2}$ faces. So, by the same reasoning, every hypercut has at most ${n \choose 3} - \left({n-1 \choose 2}-1\right)$ faces. A hypercut of this size (if one exists) is called {\em perfect}. We show that $\Q$-perfect hypercuts exist for certain integers $n$, and if a well-known conjecture by Artin\footnote{There is strong evidence for this conjecture. In particular it follows from the generalized Riemann Hypothesis.} in number theory is true, there are {\em infinitely many} such $n$. The construction is based on the $2$-complex of length-$3$ arithmetic progressions in $\Z_n$, and is of an
independent interest.

Over the field $\f_2$, surprisingly, the situation changes. There are no perfect hypercuts for $n>6$,
and the largest possible hypercut has ${n \choose 3} -  \frac{3}{4} n^2 - \Theta(n)$ faces. We completely describe all the extremal hypercuts.

Staying with $\f=\f_2$ and with $d>2$ the situation depends on the {\em parity} of $d$. As we show, for $d$ even the largest $d$-hypercuts have ${n \choose d+1}\cdot\left(1-o_n(1)\right)$ $d$-faces. When $d$ is odd, all $d$-hypercuts have density that is bounded away from 1. 
%

Equivalently, this subject can be viewed from the perspective of shadows of acyclic complexes. Thus the complement of a perfect $2$-hypercut over $\Q$ is an almost-hypertree (i.e. acylic complex with ${n-1 \choose 2}-1$ $2$-faces) with an empty shadow. Our results over $\f_2$ can be restated as saying that the least possible size of the shadow of a $2$-dimensional $\f_2$-almost-hypertree is $\frac{n^2}{4}+\Theta(n)$.

Many questions suggest themselves: Let $X$ be an $\f$-acyclic $d$-dimensional $n$-vertex simplicial complex with a full skeleton and a given size. What is the smallest possible shadow of such $X$? More specifically, what is the largest possible size of $X$ if it is shadowless?

One construction that we present here applies to all fields at once, since it is based on the combinatorial notion of collapsibility. This is a collapsible $2$-complex with $f={n-1 \choose 2} - (n+1)$ $2$-faces which remains collapsible after the addition of any other $2$-face. This yields, for every field $\f$, a shadowless $\f$-acyclic $2$-complex with $f$ $2$-faces

We note that Bj\"orner and Kalai's work \cite{BK} determines the {\em largest} possible shadow of an acyclic $d$-complex with $f$ faces. The extremal examples are maximal acyclic subcomplexes of a shifted complex with $f$ faces.

The rest of the paper is organized as following. In Section~\ref{section:combin} we introduce some additional notions in the combinatorics of simplicial complexes. Section~\ref{sec:Q} deals with the problem of largest $2$-hypercuts over $\Q$.
In Section~\ref{sec:f2} we study the same problem over $\f_2$.
In Section~\ref{sec:f2_even_dim} we construct large $d$-hypercuts over $\f_2$ for even $d \ge 4$.
In Section~\ref{sec:cl_ac} we deal with large acyclic shadowless $2$-dimensional sets.
Lastly, in Section~\ref{section:open} we present some of the many open questions in this area.
%
\section{Additional Notions and Facts from Simplicial Combinatorics}
\label{section:combin}
Recall that we view the $d$-boundary operator as a linear map over
$\f$,  that maps
vectors supported on oriented $d$-simplices to vectors supported on
$(d-1)$-simplices, given explicitly by the matrix $\partial_d$, as  defined in the Introduction.

The right kernel of $\partial_d$ is the linear space of {\em $d$-cycles}.
The left image of $\partial_d$ is the linear space $B^d(X)$ of $d$-coboundaries of $X$.
With some abuse of notation we occasionally call a set of $d$-simplices
a cycle or a coboundary if it is the {\em support} of a cycle or a coboundary.
Clearly over $\f_2$, this makes no difference. In this case each $d$-coboundary is
associated with a set $A$ of $(d-1)$-faces, and consists of those $d$-faces whose boundary has an odd intersection with $A$.

A $d$-coboundary is called {\em simple} if its support does
not properly contain the support of any other non-empty $d$-coboundary. As observed e.g., in \cite{V-con}, a coboundary is simple if and only if its support is
a hypercut. 

If $\sigma$ is a face in a complex $X$, we define its {\em link} via
$\link_\sigma(X)=\{\tau\in X : \tau \cap \sigma= \emptyset, ~ \tau\cup\sigma\in X\}$. This is clearly a simplicial complex. For instance,
the link of a vertex $v$ in a graph $G$ is $v$'s neighbour set which we
also denote by $N_G(v)$ or $N(v)$.  For a $2$-coboundary $C$ over $\f_2$ and a vertex
$v \in [n]$, it is easy to see that the graph $\link_v(C)$ generates $C$, i.e. $C = \link_v(C) \cdot \partial_2$. Namely, the characteristic vector of the $2$-faces of $C$ equals to the vector-matrix left product of the characteristic vector of the edges of $\link_v(C)$ with the boundary matrix $\partial_2$. We recall a necessary and
sufficient condition that $G = \link_v(C)$ generates a $2$-hypercut $C$
rather than a general  coboundary.

Two incident edges $uv$,$uw$ in a graph $G=(V,E)$ are said to be
$\Lambda$-{\em adjacent} if $vw\notin E$. We say that $G$ is
$\Lambda$-{\em connected} if the transitive closure of the
$\Lambda$-adjacency relation has exactly one class.
\begin{proposition}\label{prop:2cut}\cite{V-con}
A $2$-dimensional coboundary $B$ is a hypercut if and only if  the
graph $\link_v(B)$ is $\Lambda$-connected for every $v$.
\end{proposition}

\section{Shadowless Almost-Hypertrees Over $\Q$}
\label{sec:Q}
The main result of this section is a construction of $2$-dimensional {\em shadowless} $\Q$-almost-hypertrees. As mentioned above, the complement of such a complex is a {\em perfect}  hypercut having ${n\choose 3}-{n-1\choose 2} + 1$ faces which is the most possible.

\begin{theorem}\label{thm:noShadow}
  Let $n\ge 5$ be a prime for which $\Z_n^*$ is generated by $\{- 1,
  2\}$. Let $X=X_n$ be a 2-dimensional simplicial complex on
  vertex set $\Z_n$ whose $2$-faces are arithmetic progressions of
  length $3$ in $\Z_n$ with difference not in $\{ 0,\pm 1\}$ . Then,
\begin{itemize}
\item  $X_n$ is $2$-collapsible, and hence it is an
almost-hypertree over every field.
\item  $SH(X_n)=\emptyset$ over $\Q$. Consequently, the complement of $X_n$ is a
perfect hypercut over $\Q$.
\end{itemize}
\end{theorem}
The entire construction and much of the discussion of $X_n$ is carried out over $\Z_n$.
However, in the following discussion, the boundary operator $M_X$ of
$X_n$ is considered over the rationals.

We start with two simple observations. First, note
that $X_n$ has a full $1$-skeleton, i.e., every edge is contained in some $2$-face of $X$.

Also, we note that the choice of omitting the arithmetic triples with difference $\pm 1$ is completely arbitrary. For every $a \in \Z_n^*$, the automorphism $r \mapsto ar$ of $\Z_n$ maps $X_n$
to a combinatorially isomorphic complex of arithmetic triples over $\Z_n$, with
omitted difference $\pm a$. Consequently, Theorem~\ref{thm:noShadow} holds equivalently for any difference that we omit. In what follows we indeed assume for
convenience that the missing difference is not $\pm 1$, but rather $\pm 2^{-1} \in \Z_n$.

For $d \in Z^*_n$, define $E_d \;=\; E_{d, n} \;=\; \left(\,
  (0,d),(1,d+1),\ldots,(n-1,d+n-1)\, \right),$ where all additions are $\bmod ~n$.
This is an ordered subset of directed edges in $X_n$.

Similarly, we consider the collection of arithmetic triples of difference $d$,
$$F_d \;=\; F_{d, n} \;=\; \left( \, (0,d,2d),(1,d+1,2d+1),\ldots,(n-1,d+n-1,2d+n-1) \,\right) .$$

Clearly every directed edge appears in exactly one $E_{d}$ and then its reversal is in
$E_{-d}$. Likewise for arithmetic triples and the $F_d$'s.
Since we assume that
$Z^*_n$ is generated by $\{-1, 2\}$, it follows that
the powers $\left\{ 2^i \right\} \subset Z_n^*$, ${i=0,\ldots,
  {\frac{n-1}{2} - 1}}$, are all distinct, and, moreover, no power is
an additive inverse of the other.  Therefore, the sets $\left\{
  E_{2^i} \right\}$, ${i=0,\ldots, {\frac{n-1}{2} - 1}}$, constitute a
partition of the $1$-faces of $X_n$. Similarly, the sets $\left\{ F_{2^j}
\right\}$, ${j=0,\ldots, {\frac{n-1}{2} - 2}}$, constitute a partition
of the $2$-faces of $X_n$. The omitted difference is $2^{\frac{n-1}{2} -
  1} \in \{  \pm 2^{-1} \} $, as assumed (the sign is determined according to whether
$2^{\frac{n-1}{2}} =1$ or $-1$).

\begin{lemma}\label{lem:boundaryOfArithComplex}
Ordering the rows of the adjacency matrix $M_{X}$ by $E_{2^i}$'s,  and ordering the columns by the $F_{2^i}$'s, the matrix $M_{X}$ takes the following form:
\begin{equation}
\label{eqn:matrix_collapse_form}
M_{X} ~=~
\left(\begin{matrix}
I+Q&0&0&...\\
-I&I+Q^2&0&...\\
0&-I&\ddots&...\\
0&0&\ddots&I+Q^{2^{\frac{n-1}{2} - 2}}\\
0&0&...&-I
\end{matrix}\right)
\end{equation}
where each entry is an $n\times n$ matrix (block) indexed by $\Z_n$, and $Q$ is
a permutation matrix corresponding to the linear map $b\mapsto b+1$ in $\Z_n$.
\end{lemma}
\begin{proof}
Consider an oriented face $\sigma \in F_{2^i} \subset X_n$. Then,
$\sigma=(b,b+2^i,b+2^{i+1})$ for some $b\in\Z_n$ and $0\le i\le \frac{n-1}{2} - 2$, i.e., $\sigma$ is the $b$-th element in $F_{2^i}$. By definition, $\partial \sigma=(b,b+2^i)+(b+2^i,b+2^{i+1})-(b,b+2^{i+1})$.
The first two terms in $\partial\sigma$ are the $b$-th and $(b+2^i)$-th elements in $E_{2^{i}}$ respectively;
the third term corresponds to the $b$-th element in $E_{2^{i+1}}$.  Thus, the blocks indexed by $E_{2^{i}} \times F_{2^{i}}$ are of the form $I+Q^{2^i}$, the blocks $E_{2^{i+1}} \times E_{2^{i}}$ are $-I$, and the rest is 0.
\end{proof}

We may now establish the main result of this section.
\begin{proof}{\bf (of Theorem~\ref{thm:noShadow})}~
We start with the first statement of the theorem. Let $m= \frac{n-1}{2}$.

Lemma~\ref{lem:boundaryOfArithComplex} implies that the edges
in $E_{2^{m - 1}}$ are exposed. Collapsing on these edges leads to elimination of $E_{2^{m - 1}}$ and the faces in $F_{2^{m - 2}}$. In terms of the
matrix $M_X$, this corresponds to removing the rightmost "supercolumn". Now
the edges in $E_{2^{m - 2}}$ become exposed, and collapsing them
leads to elimination of $E_{2^{m - 2}}$, and $F_{2^{m - 3}}$. This results in exposure of
$E_{2^{m - 3}}$, etc. Repeating the argument to the end, all the faces of $X_n$ get eliminated, as claimed.

To show that $X_n$ is an almost-hypertree we need to show that the number of its $2$-faces is $\binom{n-1}{2} - 1$. Indeed,
\[
|X_n| ~=~ \sum_{j=0}^{\frac{n-1}{2} - 2} |F_{2^j}| ~=~ \left( \frac{n-1}{2} - 1 \right) \cdot n ~=~
\binom{n-1}{2} - 1\;.
\]
We turn to show the second statement of the theorem, i.e., that $SH(X_n)=\emptyset$. Let $u\in \Q^{\binom{n}{2}}$, be a vector indexed by the edges of $X_n$, where $u_e=2^i$ when $e\in E_{2^i}$. Here we think of $2^i$ as an integer (and not an element in $\Z_n$). We claim that for every $\sigma\in K_n^2$,
\[
\langle u,\partial\sigma \rangle\,=\,0 ~~\iff~~ \sigma\in X_n\,.
\]

Indeed, for every face $\sigma\in \Kn$, exactly three coordinates in the vector $\partial\sigma$ are non-zero, and they are $\pm 1$. Since the entries of $u$ are successive powers of $2$, the condition $\langle u,\partial\sigma \rangle\,=\,0$ holds iff $\partial\sigma$ (or $-\partial\sigma$) has two $1$'s in $E_{2^i}$ and one $-1$ in $E_{2^{i+1}}$ for some
$0\le i\le \frac{n-1}{2} - 1$.  This happens if and only if $\sigma$
is of the form $(b, b+2^i, b+2^{i+1})$, i.e., precisely when $\sigma \in X_n$.

This implies that $X_n$ is closed, i.e. $SH(X_n)=\emptyset$, since any 2-face $\sigma \in \Kn$ spanned by $X_n$ must satisfy $\langle u,\partial\sigma \rangle =0$, this being precisely the characterisation of $X_n$.
Thus, $X_n$ is a closed set of co-rank 1. Therefore, its complement is a hypercut. Moreover, since
$X_n$ is almost-hypertree, this hypercut is perfect.

%
\end{proof}

When the prime $n$ does not satify the assumption of Theorem~\ref{thm:noShadow} we can still say something about the structure of $X_n$. Let the group $G_n=\Z_n^*\slash\{\pm 1\}$, and let $H_n$ be the subgroup of
$G_n$ generated by $2$. Then,
\begin{theorem}
For every prime number $n$, $\rank_\Q (X_n) = |X_n| -(n-1)\cdot ([G_n:H_n] - 1)$.~ In particular,
$X_n$ is acyclic if and only if $Z_n^*$ is generated by $\{\pm 1, 2\}$.
\end{theorem}

We only sketch the proof. We saw that the partition of $X_n$'s edges and faces to the sets $E_i$ and $F_i$. We consider also a coarser partition by joining together all the $E_i$'s and $F_i$'s for which $i$ belongs to some coset of $H_n$. This induces a block structure on $M_X$ with $[G_n:H_n]$ blocks. An argument as in the proof of Lemma \ref{lem:boundaryOfArithComplex} yields the structure of these blocks. Finally, an easy computation shows that one of these blocks is $2$-collapsible, and each of the others contribute precisley $n-1$ vectors to the right kernel.

We conclude this section by recalling the following well-known conjecture of Artin which is implied by the generalized Riemann hypothesis ~\cite{artin}.
\begin{conjecture}[Artin's Primitive Root Conjecture]
Every integer other than -1 that is not a perfect square is a primitive root modulo infinitely many primes.
\end{conjecture}

This conjecture clearly yields infinitely many primes $n$ for which $\Z_n^*$ is generated by $2$. (It is even conjectured that the set of such primes has positive
density). Clearly this implies that the assumptions of
Theorem~\ref{thm:noShadow} hold for infinitely many primes $n$.
%
%

\section{Largest Hypercuts over $\f_2$}
\label{sec:f2}
In this section we turn to discuss our main questions over the field $\f_2$. The main result of this section is:

\begin{theorem}
  \label{thm:main}
For large enough $n$, the largest size of a $2$-dimensional hypercut over $\f_2$ is ${n \choose 3}
- \left( \frac{3}{4} n^2 - \frac{7}{2}n + 4 \right)$ for even $n$ and ${n \choose 3}
- \left(\frac 34n^2-4n+\frac{25}{4} \right)$ for odd $n$.
\end{theorem}

\begin{remark}
The proof provides as well a characterization of all the extremal cases of this theorem.
\end{remark}

Since no confusion is possible, in this section we use the shorthand term {\em cut} for a $2$-dimensional hypercut.

The first step in proving Theorem \ref{thm:main} is
the slightly weaker Theorem \ref{thm:1}. A further refinement yields the
tight upper bound on the size of cuts.

Note that since ${[n] \choose 3}$ is a coboundary, the complement
  $\bar{C}= {[n] \choose 3} \setminus C$ of any cut $C$ is a
  coboundary. Moreover, the complement of the $(n-1)$-vertex graph,
  $\link_v(C)$, is a link of $\bar{C}$. In what follows, $\link_v(C)$
  is always considered as an $(n-1)$-vertex graph with vertex set $[n] \setminus
  \{v\}$. Occasionally, we will consider the graph $\link_v(C) \cup \{v\}$ which has $v$ as an isolated vertex.

\begin{theorem}
  \label{thm:1}
The size of every $n$-vertex cut is at most ${n \choose 3} - \frac{3}{4}\cdot n^2  + o(n^2)$. In every cut $C$ that attains this bound there is a vertex $v$ for which the graph $G = \link_v(C)$ satisfies either
\begin{enumerate}
\item $\bar{G}$ has one
vertex of degree $\frac n2 \pm o(n)$ and all other vertices have
degree $o(n)$. Moreover, $|E(G)| = n-1 + o(n)$.
\item
$\bar{G}$ has one vertex of degree $n - o(n)$, one vertex of degree $\frac n2 \pm o(n)$, and all other vertices have
degree $o(n)$. Moreover $|E(G)| = 2n \pm o(n)$.
\end{enumerate}
\end{theorem}
We need to make some preliminary observations.

\begin{observation}
  \label{obs:zero}
  Let $G= (V,E)$ be a graph with $n$ vertices, $m$ edges and $t$ triangles and let $C$ be the coboundary generated by $G$. Then
$|C| = nm - \sum_{v \in V} d_v^2 +  4t$.
\end{observation}
\begin{proof}
  Let $e=(u,v) \in E(G)$. Then $\link_e(C)$ consists of those vertices $x\neq u, v$ that are adjacent to both or none of $u, v$. Namely, $|\link_e(C)| = n - d_u - d_v +2|N(v) \cap N(u)|$. Clearly $|N(v) \cap N(u)|$ is the number of triangles in $G$ that contain $e$. But $\sum_{e \in E(G)}|\link_e(C)|$ counts every two-face in $C$ three times or once, depending on whether or not it is a triangle in $G$. Therefore $$|C| +2t =\sum_{(u,v) \in E} \left(n - d_u - d_v +2|N(v) \cap
N(u)|\right).$$
The claim follows.
\end{proof}

Two vertices in a graph are called {\em clones} if they have the same set of neighbours (in particular they must be nonadjacent).

\begin{observation}
  \label{obs:1991}
  For every nonempty cut $C$ and  $x\in V$ the graph
  $\link_x(\bar C)$ is
connected and has no clones, or it contains one isolated vertex and a
complete graph on the rest of the vertices.
\end{observation}
\begin{proof}
  Directly follows from the fact that $\link_x(C)$ is $\Lambda$-connected (Proposition \ref{prop:2cut}).
\end{proof}

The size of a cut $C$ for which $\link_x(\bar{C})$ is the union of a complete graph on $n-2$ vertices and an isolated vertex equals to $n-2$, which is much smaller than the bound in Theorem \ref{thm:1}. We restrict the following discussion to cuts $C$ for which
$\bar{G}=\link_x(\bar{C})$ is connected and has no clones. Let $V=V(\bar G)$, and we denote by $N(v):=N_{\bar G}(v)$. For every $S\subseteq V$, an {\em $S$-atom} is a subset $A\subseteq V$ which satisfies: $(u,v)\in E(\bar  G) \iff (u',v)\in E(\bar G)$ for every $u,u'\in A$ and $v\in S$.

The next claim generalizes Observation \ref{obs:1991}.

\begin{claim}
  \label{cl:167}
Suppose $C$ is a cut and $G=(V,E) = \link_x(C)$ for some vertex $x\notin V$. Let
$S\subseteq V$, and $G' = \bar{G} \setminus S$. Then, for every non-empty $S$-atom $A$, at least $|A|-2$ of the edges in $G'$ meet $A$.
\end{claim}
\begin{proof}
Let $H$ be the subgraph of $G'$ induced by an atom $A$. If $H$ has at most two
connected components, the claim is clear, since a connected graph on
$r$ vertices has at least $r-1$ edges. We next consider what happens if $H$ has
three or more connected components. We show that every component except
possibly one has an edge in $\bar E$ that connects it to $V\setminus
(S\cup A)$. This clearly proves the claim.

So let $C_1, C_2, C_3$ be connected components of $H$, and suppose
that neither $C_1$ nor $C_2$ is connected in $\bar G$ to $V\setminus
(S\cup A)$. Let $F:= \cup_{1\le i< j\le 3} C_i\times C_j\subseteq
E$. Since $G$ is $\Lambda$-connected, there must be a $\Lambda$-path
connecting every edge in $C_1\times C_2$ to every edge in $C_2\times
C_3$. However, every path that starts in $C_1\times C_2$ can never
leave it. Indeed, let us consider the first time this $\Lambda$-path
exits $C_1\times C_2$, say $xy$ that is followed by $yw$, where $x\in
C_1, y\in C_2$, $w \notin C_1\cup C_2$ and $yw \notin E$.  By the atom condition, a vertex in $S$ does not distinguish between vertices $x, y\in A$, whence $w\notin S$. Finally $w$ cannot be in $A$, for $xw\notin E$ would imply that $w\in C_1$. Hence, $C_1$ is connected in $\bar G$ to $V\setminus (S\cup A)$, a contradiction.
\end{proof}

In the following claims, let $G=(V,E) = \link_x(C)$, for a cut $C$, and $x \notin V$, and let $\bar{G} = (V,\bar{E}) = \link_x(\bar{C})$.
Denote by $d=(d_1 \geq d_2 \geq \ldots \geq d_{n-1} \geq 1)$ the sorted degree sequence of $\bar G$. We label the vertices $v_1, \ldots, v_{n-1}$ so that $d(v_i) = d_i$ for all $i$.

\begin{claim}
  \label{cl:185}
$d_1 \leq m/2 +1$.
\end{claim}
\begin{proof}
Apply Claim \ref{cl:167} with $S=\{v_1\}$ and $A=N(v_1)$. It yields
the existence of at least $|A|-2$ edges in $\bar{G}$ that meet $A$ but
not $v_1$. Since $|A|=d_1$, $m\ge d_1+(d_1-2)$, implying the claim.
\end{proof}

\begin{claim}
  \label{cl:12}
$d_1 + d_ 2 \leq \frac{m+n}{2}$.
\end{claim}
\begin{proof}
Apply Claim \ref{cl:167} with $S=\{v_1,v_2\}$ and $A=N(v_1) \cap
N(v_2)$ to conclude that $m\ge d_1+d_2+|A|-3$ (as $(v_1,v_2)$ might be
an edge). By inclusion-exclusion, $n-3\ge d_1+d_2-|A|$. These two inequalities imply the claim.
\end{proof}

\begin{claim}
  \label{cl:171}
For every $k$,  $\sum_{i=1}^k d_i \leq m- \frac{n}{2} + \frac{k^2}{2}+ 2^{k}$.
\end{claim}

\begin{proof}
  There are at most $2^k$ atoms of $S=\{v_1,...,v_k\}$, and we apply
  Claim \ref{cl:167} to each of them. There are at least $|A|-2$ edges
  with one vertex in atom $A$ and the other vertex not in
  $S$. Consequently, there are at least $\frac12(n-k-2\cdot 2^k)$
  edges in $\bar{G} \setminus S$ (as each edge may be counted twice).
In addition, there may be at most ${k \choose 2}$ edges induced by $S$,
  hence the total number of edges is bounded by
$$
m \ge \sum_{i=1}^{k} d_i-{k \choose 2} + \frac12(n-k-2\cdot 2^k)
$$
\end{proof}
\begin{proof} {\bf of Theorem \ref{thm:1}}
  Let $C$ be a cut, and assume by contradiction that $|\bar{C}| \leq
  \frac{\gamma}{3} n^2$, for some $\gamma < \frac 94$. By averaging, there is a
  link, say $\bar{G} = (V,\bar E) = \link_v(\bar{C})$, of at most $\frac{3|\bar C|}{n}\leq \gamma n <
  9n/4$ edges, where $V = [n]\setminus \{v\}$. Let $|\bar E|= m \leq
\gamma n$ for some $\gamma < 9/4$.
We will show that $|\bar{C}| \geq 3n^2/4 - o(n^2)$ contradicting the
assumption.

Indeed, Observation \ref{obs:zero} implies that $|\bar{C}| = mn - \sum_{v \in \bar{G}} d_v^2 +
  4t$ where $t$ is the number of triangles in $\bar{G}$. Hence it
  suffices to show that $s(\bar{C}):= mn - \sum_v d_v^2 \geq
  \frac{3}{4} \cdot n^2 - o(n^2)$.

  Given a sequence of reals $d=(d_1 \geq d_2 \geq \ldots \geq d_{n-1}
  \geq 1)$, we denote $f(d) := mn - \sum_i d_i^2$ where $m
  =\frac{1}{2} \sum_i d_i$. With this notation $s(\bar{C}) = f(d)$
  where $d$ is the sorted degree sequence of $\bar{G}$ and $v_1,
  \ldots v_{n-1}$ the
  corresponding ordering of the vertices.
  I.e., $d(v_i) = d_i \geq d_{i+1}$.

We want to reduce the problem of proving a lower bound on $f(d)$ to
showing a lower bound on $f_k(d) =mn -\sum_1^k d_i^2$, where $k =
k(n)$ is an appropriately chosen slowly growing function. Clearly
$f_k(d)-f(d)=\sum_{j=k+1}^{n-1} d_j^2$. But $d_j\le \frac{2m}{j}$ for
all $j$ whence $\sum_{j=k+1}^{n-1} d_j^2 \le 4m^2
\sum_{j=k+1}^{\infty}\frac{1}{j^2}<\frac{4m^2}{k}$, i.e, $f_k(d) \leq
f(d) +\frac{4m^2}{k}$. Since $m<\frac 94 n$, it suffices to show that
$f_k(d) \geq \frac{3}{4} n^2 - o(n^2)$ for an arbitrary
$k=\omega_n(1)$, which is our next goal.

We first note that.
\begin{claim}\label{cl:1993}
  \label{cor:21} For every $k =o(\log n)$,
$f_k(d) \geq mn - d_1^2 - (m - n/2 - d_1)\cdot d_2 - o(n^2)$.
\end{claim}
\begin{proof}
  $f_k(d) =mn -\sum_1^k d_i^2\ge mn - d_1^2-(\sum_2^k d_i)\cdot d_2
  \geq mn - d_1^2 - (m - n/2 - d_1)\cdot d_2 - o(n^2)$, where the
  second step is by convexity, and the last
  step uses Claim~\ref{cl:171}.
\end{proof}

We now normalize everything in terms of $n$, namely, write
$$m = \gamma \cdot n, ~~~d_1 = x \cdot n, ~~~d_2 = y \cdot n.$$
$$
g(\gamma,x,y) :=\frac{1}{n^2} f_k - o(1) = \gamma - x^2 - \gamma \cdot y + \frac{y}{2} +  xy,
$$

The problem of minimizing $f_k$ subject to our assumptions on $\gamma$, $d_2 \leq d_1$, and Claims \ref{cl:185}, \ref{cl:12}, \ref{cl:171} becomes
\vspace{0.2in}

{\bf Optimization problem A}

Minimize $g(\gamma,x,y)$, subject to:
\begin{enumerate}
\item $1\le \gamma \le \frac 94.$
\item $0\le y\le x \le \min\left(\frac \gamma 2,1\right).$
\item $x+ y  \leq \gamma - \frac{1}{2}$.
\item $x+y\le\frac{1+\gamma}{2}.$
\end{enumerate}

This problem is answered in the following Theorem whose proof is in the appendix.

\begin{theorem}\label{thm:5}
The answer to Optmization problem A is $\min g(\gamma,x,y) = \frac{3}{4}$. The optimum
is attained in exactly two points $(\gamma=1, x=\frac{1}{2}, y = 0)$ and $(\gamma
= 2, x=1, y=\frac{1}{2})$.
\end{theorem}
Plugging the optimal values on $\gamma, x,y$ back into Claim
\ref{cl:1993} completes the proof of the Theorem.
\end{proof}

\begin{proof}[ Proof of Theorem \ref{thm:main}]
Let us recall some of the facts proved so far concerning the largest $n$-vertex $2$-hypercut $C$. Pick an arbitrary vertex $v$. Since $C$ is a coboundary, it can be generated by an $n$-vertex graph which consists of the isolated vertex $v$, and  $G=\link_v(C)$, an $(n-1)$-vertex $\Lambda$-connected graph. Similarly, $\bar C$ can be generated by the disjoint union of ${v}$ and $\bar G$. As we saw, there exists some ${v}$ for which the corresponding $\bar{G}$ satisfies either

$$ {\bf CASE ~(I)}:~~~~~~~m = n-1+o(n), ~d_1 =\frac n2 \pm o(n)\text{~and~} d_2 = o(n),$$
or
$$ {\bf CASE ~(II)}:~~~~~~m = 2n \pm o(n), ~d_1 =n - o(n), ~d_2 =\frac n2 \pm o(n) \text{~and~}d_3 = o(n).$$

where, as before, $m=|E(\bar G)|$, $d_1\ge d_2 \ge \dots \ge d_{n-1}$ is the degree sequence of $\bar G$, with $d_i=d(v_i)$. We denote by $t$ the number of triangles in $\bar G$. Since $C$ is the largest cut, the graph $\bar G$ attains the minimum of $f(\bar G)=nm-\sum d_i^2+4t$ among all graphs whose complement is $\Lambda$-connected.

We now turn to further analyse the structure of $\bar G$, in {\bf CASE (I)}.
\begin{lemma}
\label{lm:str_barG_caseI}
Suppose that $\bar G$ satisfies~ {\bf CASE (I)} and let $H=\bar G\setminus v_1$. Then $H$ is either (i) A perfect matching, or (ii) A perfect matching plus an isolated vertex, or (iii) A perfect matching plus an isolated vertex and a 3-vertex path.
\end{lemma}
\begin{proof}
The proof proceeds as follows: for every $H$ other than the above, we find a local variant $\bar G_1$ of $\bar G$ with $f(\bar G_1)< f(\bar G)$. We then likewise modify $G_1$ to $G_2$ etc., until for some $k\ge 1$ the graph $G_k$ is $\Lambda$-connected. The process proceeds as follows.

For every connected component $U$ of $H$ of even size $|U|\ge 4$, we replace $H|_U$ with a perfect matching on $U$, and connect $v_1$ to one vertex in each of these $\frac{|U|}{2}$ edges. {\em Now all connected components of $H$ are either an edge or have an odd size}.

Consider now odd-size components.  Note that $H$ can have at most one isolated vertex. Otherwise $\bar G$ is disconnected or it has clones, so that $G$ is not $\Lambda$-connected.
As long as $H$ has two odd connected components which together have $6$ vertices or more, we replace this subgraph with a perfect matching on the same vertex set, and connect $v_1$ to one vertex in each of these edges. If the remaining odd connected components are a triangle and an isolated vertex, remove one edge from the triangle, and connect $v_1$ only to one endpoint of the obtained $3$-vertex path. In the last remaining case $H$ has at most one odd connected component $U$.

If no odd connected components remain or if $|U|=1$, we are done.

In the last remaining case $H$ has a single odd connected component of order $|U|\ge 3$. We replace $H|_U$ with a matching of $(|U|-1)/2$ edges, connect $v_1$ to one vertex in each edge of the matching and to the isolated vertex. If, in addition, there is a connected component of order $2$ with both vertices adjacent to $v_1$ (Note that by the proof of Claim \ref{cl:167} there is at most one such component.), we remove as well one edge between $v_1$ and this component.

All these steps strictly decrease $f$. We show this for the first kind of steps. The other cases are nearly identical.

Recall that $|E(H)|=\frac n2\pm o(n)$ and that $H$ has at most one isolated vertex. Therefore every connected component in $H$ has only $o(n)$ vertices. Let $U$ be a connected component with $2u\ge 4$ vertices of which $0<r\le 2u$ are neighbours of $v_1$, and let $\beta = |E(H|_U)| - (2u-1)\ge 0$. Let $\bar G'$ be the graph after the aforementioned modification w.r.t. $U$. We denote its number of edges and triangles by $m'$ and $t'$ resp., and its degree sequence by $d_i'$. Then,
\begin{eqnarray*}
f(\bar G)-f(\bar G') = n(m-m')-\sum_i(d_i^2-d_i'^2)+4(t-t') \ge \\
n(\beta+r-1)-\left(d_1^2-(d_1-r+u)^2\right)-\sum_{i\in U}d_i^2 \ge \\
n(\beta+r-1)+\left(u-r\right)(2d_1+u-r)-2u(4u-2+2\beta+r).
\end{eqnarray*}
In the second row we use $t\ge t'$, which is true since the modification on $U$ creates no new triangles. In the third row we use $\sum_{i\in U}d_i^2\le(\max_{i\in U}d_i)\left(\sum_{i\in U}d_i\right).$

Let us express $d_1=\frac{n-w}{2}$ where $w=o(n)$. What remains to prove is that
\begin{eqnarray*}
n(\beta+u-1)+(u-r)(u-r-w)\ge 2u(4u-2+2\beta+r).
\end{eqnarray*}
Or, after some simple manipulation, and using the fact $r \leq 2u$, that
\[
\beta n+(u-1)n\ge 4\beta u+u(7u+w+3r-4).
\]
This is indeed so since $u=o(n)$ implies that $\beta n\gg\beta u$ and $2\le u\le o(n)$ implies $(u-1)n\gg u(7u+w+4r-4)$.

The other cases are done very similarly, with only minor changes in the parameters. In the case of two odd connected components which together have $2u\ge 6$ vertices, in the final step the main term is $(\beta + u -2)n \ge n+\beta u$ since $u>2$. In the case of changing a triangle to a 3-vertex path the main term in the final inequality is $(\beta+u-1)n=n$.

\end{proof}

The structure of $\bar G$ for~ {\bf CASE (I)} is almost completely determined by Lemma \ref{lm:str_barG_caseI}. Since $G$ is $\Lambda$-connected, in $\bar G$ $v_1$ must have a neighbour in each component of $H$, and can be fully connected to at most one component. In addition, if $P$ is a 3-vertex path in $H$, then $v_1$ has exactly one neighbour in $P$ which is an endpoint. Otherwise we get clones. Therefore the only possible graphs are those that appear in Figure 1. The first row of the figure applies to odd $n$, where the optimal $\bar G$ satisfies $f(\bar G)=\frac 34n^2-4n+\frac{25}{4}$. The other rows correspond to $n$ even, with four optimal graphs that satisfy $f(\bar G)=\frac 34n^2-\frac 72n+4$.

\begin{figure}[h!]
\label{fig:caseI}
  \centering
    \includegraphics[width=1\textwidth]{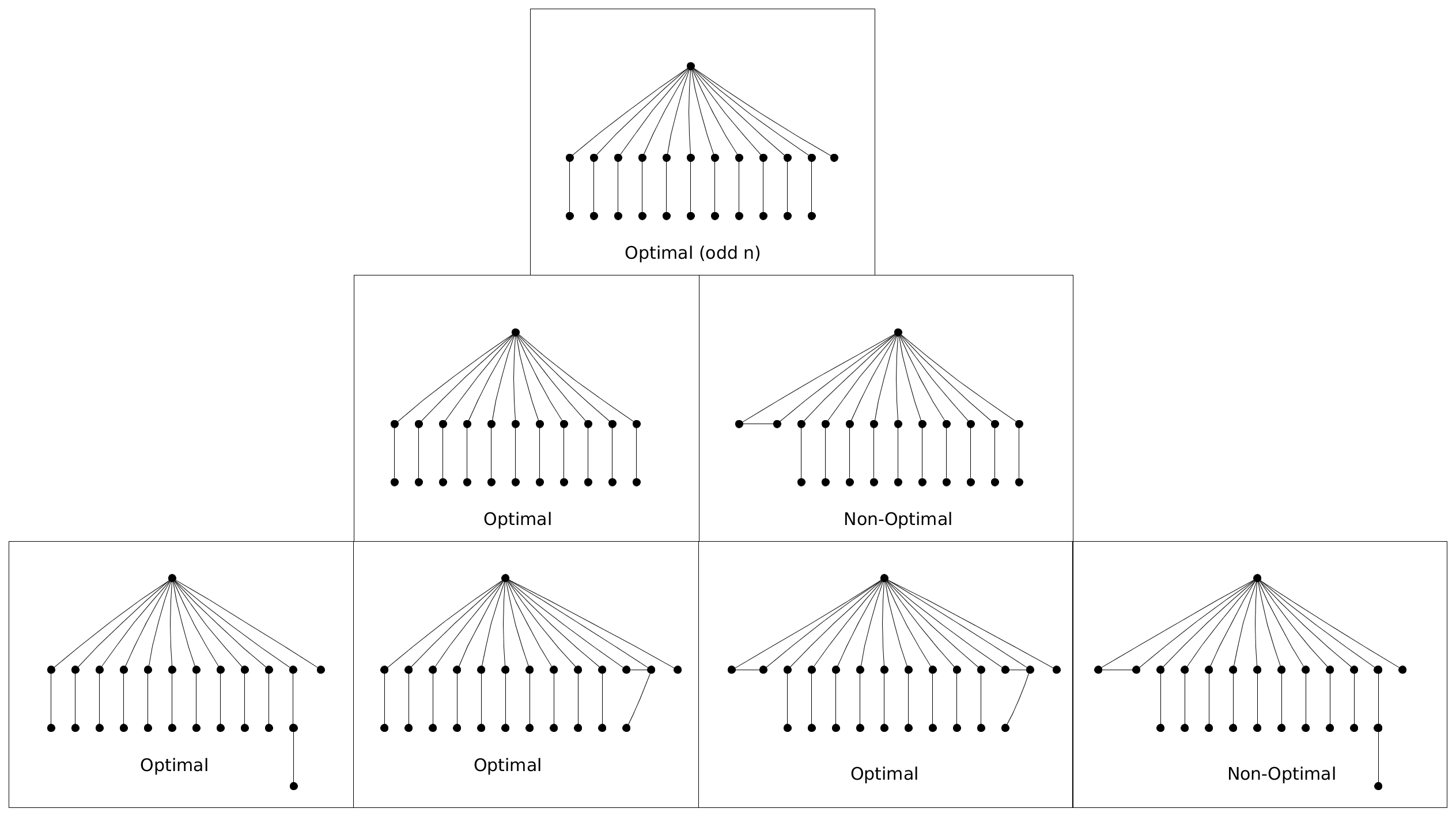}
    \caption{The graphs $\bar G$ that are considered in the final stage of the proof of~ {\bf CASE (I)}. The first row refers to the only possibility for odd $n$. The second row to even $n$, where $H$ is a perfect matching. The third row refers to even $n$, where $H$ a disjoint union of an isolated vertex, 3-path and a matching.}
\end{figure}

This concludes~ {\bf CASE (I)}, and we now turn to {\bf Case (II)}. Our goal here is to reduce this back to {\bf CASE (I)}, and this is done as follows.

\begin{claim}
\label{clm::redII2I}
Let $\bar G = \link_v(\bar{C})$ be a graph on $n-1$ vertices with  parameters as in {\bf CASE (II)}. If $H=\bar G\setminus\{v_1,v_2\}$ has an isolated vertex $\rm{z}$ that is adjacent in $\bar G$ to $v_1$ then $f(\bar G)$ is bounded by the extremal examples found in {\bf CASE (I)}.
\end{claim}
\begin{proof}
Let $S$ be the star graph on vertex set $V\cup\{v\}$ with vertex $v_1$ in the center and $n-1$ leaves. Consider the graph $F := \bar{G} \xor S$ on the same vertex set, whose edge set is the symmetric difference of $E(S)$ and $E(\bar G)$. Since every triplet meets $S$ in an even number of edges, the coboundary that $F$ generates equals to the coboudnary that $\bar G$ generates, which is  $\bar C$.

In addition, $\rm{z}$ is an isolated vertex in $F$ since its only neighbor in $\bar G$ is $v_1$. Consequently, $F  = \link_{\rm z}(\bar{C})$, and the claim will follow by showing that $F$ agrees with the conditions of {\bf CASE (I)}.
Indeed, $\mbox{deg}_F(v_1) =
n-1 - \mbox{deg}_{\bar{G}}(v_1) = o(n)$, and $|\mbox{deg}_F(u)-\mbox{deg}_{\bar G}(u)|\le 1$ for every other vertex $u$.  Hence, $\mbox{deg}_F(v_2) =
\frac{n}{2} \pm o(n)$, and $\mbox{deg}_F(u) =o(n)$ for every other vertex $u$.
\end{proof}

If $H$ has no isolated vertex that is adjacent in $\bar G$ to $v_1$, we show how to modify $\bar G$ to a graph $\bar G_1$, such that (i) $G_1$ is $\Lambda$-connected, (ii) $\bar G_1\setminus\{v_1,v_2\}$ has an isolated vertex which is adjacent to $v_1$ in $\bar G_1$ , and (iii) $f(\bar G_1)<f(\bar G)$.

Since $G$ is $\Lambda$-connected and using the proof of claim \ref{cl:167}, $H$ has at most one connected component $U_1$ in $H$ where all vertices are adjacent to $v_1$ and not to $v_2$ in $\bar G$. Similarly, it has at most one connected component $U_2$ where all vertices are adjacent to both $v_1$ and $v_2$. Also, since $d_1=n-o(n)$, $d_2=\frac n2 \pm o(n)$ and $H$ has at most 3 isolated vertices, there exists an edge $xy \in E(H)$ such that $xv_1, xv_2, yv_1\in E(\bar{G})$, but $yv_2\not\in E(\bar{G})$.

$G_1$ is constructed as following:
\begin{enumerate}
\item If neither components $U_1$ nor $U_2$ exist, remove the edge $xy$ and the edge $v_1v_2$, if it exists. Otherwise, let $r:=|U_1\cup U_2|$.
\item If $r$ is even, replace it in $H$ with a perfect matching on $u-2$ vertices and two isolated vertices. Connect $v_1$ to every vertex in $U_1\cup U_2$. Make $v_2$ a neighbor of one of the isolated vertices, and one vertex in each of the edges of the matching. Additionally, remove the edge $v_1v_2$ if it exists.
\item If $u$ is odd, replace it in $H$ by a perfect matching on $u-1$ vertices and one isolated vertex. Connect $v_1$ to every vertex in $U_1\cup U_2$, and $v_2$ to one vertex in each edge of the matching.
\end{enumerate}

The fact that value of $f$ decreased is shown similarly to the
calculation in {\bf CASE (I)}. 
\end{proof}

\section{Large $d$-hypercuts over $\f_2$ in even dimensions}
\label{sec:f2_even_dim}
In this section we consider large $d$-dimensional hypercuts over $\f_2$ for $d>2$. We show that for $d$ even the largest $d$-hypercuts have ${n\choose d+1}\left(1-o_n(1)\right)$ $d$-faces. In contrast, for odd $d$ we observe that the density of every $d$-hypercut is bounded away from $1$.
\begin{theorem}\label{thm:f2_even_dim}
For every $d\ge 2$ even there exists an $n$-vertex $d$-hypercut over $\f_2$ with ${n\choose d+1}\left(1-o_n(1)\right)$ $d$-faces.
\end{theorem}

Before we prove the theorem, let us explain why the situation is different in odd dimensions. Recall that Turan's problem (e.g., ~\cite{keevash_survey}) asks for the largest density $ex(n,K_{d+2}^{d+1})$ of a $(d+1)$-uniform hypergraph that does not contain all the hyperedges on any set of $(d+2)$ vertices. For $d$ odd a hypercut has this property, because (see Section \ref{section:combin}), the characteristic vector $C\in\f_2^{n\choose d+1}$ of a $d$-hypercut is a coboundary, i.e., $C\cdot\partial_{d+1}=0$. A simple double-counting argument shows that the density of $C$ cannot exceed $1-\frac{1}{d+2}$, and in fact, a better upper bound of $1-\frac{1}{d+1}$ is known~\cite{sidorenko}. One of the known constructions for the Turan problem yields $d$-coboundaries with density $\frac 34 - \frac {1}{2^{d+1}} - o(1)$ for $d$ odd ~\cite{de_caen}. In particular for $d=3$ this gives a lower bound of $\frac {11}{16}=0.6875$. In YP's MSc thesis~\cite{yuval_msc} an upper bound of $0.6917$ was found using flag algebras.

We now turn to prove the theorem for some $d\ge 4$ even. As before, a $d$-hypercut $C$ is an inclusion-minimal set of $d$-faces whose characteristic vector is a coboundary. In addition, every $d$-coboundary $C$ and every vertex $v$ satisfy $C = \link_v(C)\cdot\partial_d$. Recall that $C$ is a $2$-hypercut iff $\link_v(C)$ is $\Lambda$-connected for some vertex $v$. In dimension $>2$ we do not have such a charaterization, but as we show below, an appropriate variant of the {\em sufficient} condition for being a hypercut does apply in all dimensions.

Let $\tau,\tau'$ be two $(d-1)$-faces in a $(d-1)$-complex $K$. We say that they are $\Lambda$-{\em adjacent} if their union $\sigma=\tau\cup\tau'$ has cardinality $d+1$, and $\tau,\tau'$ are the only $d$-dimensional subfaces of $\sigma$ in $K$. We say that $K$ is $\Lambda$-connected if the transitive closure of the $\Lambda$-adjacency relation has exactly one class.
\begin{claim}
Let $C$ be a $d$-dimensional coboundary such that the $(d-1)$-complex $K=\link_v(C)$ is $\Lambda$-connected for some vertex $v$. Then $C$ is a $d$-hypercut.
\end{claim}
\begin{proof}
Suppose that $\emptyset\neq C' \subsetneq C$ is a $d$-coboundary and let $K'=\link_v(C')$. Note that $\emptyset\neq K' \subsetneq K$ and therefore there are $(d-1)$-faces $\tau,\tau'$ which are $\Lambda$-adjacent in $K$ such that $\tau'\in K'$ and $\tau \notin K'$. Consider the $d$-dimensional simplex $\sigma=\tau\cup\tau'$. On the one hand, since exactly two of the facets of $\sigma$ are in $K$, it does not belong to $C=K\cdot\partial_d$. On the other hand, it does belong to $C'$ since exactly one of its facets ($\tau'$) is in $K'$. This contradicts the assumption that $C'\subset C$.
\end{proof}

\begin{proof}[Proof of Theorem \ref{thm:f2_even_dim}]
We start by constructing a random $(n-1)$-vertex $(d-1)$-dimensional complex $K$ that has a full skeleton, where each $(d-1)$-face is placed in $K$ independently with probability $p:=1-n^{-\frac{1}{3d-3}}$. We show that with probability $1-o_n(1)$ the complex $K$ is $\Lambda$-connected, whence $C:=K\cdot\partial_d$ is almost surely a $d$-hypercut of the desired density. 

We actually show that $K$ satisfies a condition that is stronger than $\Lambda$-connectivity. Namely, let $\tau,\tau' \in K$ be two distinct $(d-1)$ faces. We find $\pi, \pi'$ where $\pi$ is $\Lambda$-adjacent to $\tau$ and $\pi'$ is $\Lambda$-adjacent to $\tau'$ and in addition the symmetric differences get smaller $|\pi\oplus\pi'|<|\tau\oplus\tau'|$. To this end we pick some vertices $u\in\tau\setminus\tau'$, $u'\in\tau'\setminus\tau$ and aim to show that with high probability there is some $x\notin\tau\cup\tau'$ for which the following event $P_x$ holds:

\[
\pi_x:=\tau\cup\{x\}\setminus\{u\}\text{~and~} \pi'_x:=\tau'\cup\{x\}\setminus\{u'\} \text{~are in~} K, \text{~and~}
\]
\[
\tau\text{~ is~} \Lambda-\text{~adjacent to~}\pi_x\text{~ and~} \tau'\text{~ is~}\Lambda\text{-~adjacent to~}\pi'_x
\]

In other words, it is required that $\pi_x \in K$ and $\tau\cup\{x\}\setminus\{w\}\notin K$ for every $w\in\tau\setminus\{u\}$, and similarly for $\tau',\pi_x'$. Therefore $\Pr(P_x)=p^2\cdot(1-p)^{2d-2}$. Moreover, the events $\{P_x~|~x\notin\tau\cup\tau'\}$ are independent. Hence, the claim fails for some $\tau,\tau'$ with probability at most $\left(1-p^2\cdot(1-p)^{2d-2}\right)^{n-2d}=\exp{[-\Theta(n^{1/3})]}$. The proof is concluded by taking the union bound over all pairs $\tau,\tau'$.

\end{proof}

\section{Large Collapsible $2$-dimensional Hyperforests with no Shadow}
\label{sec:cl_ac}

The main result of this section is a construction of a shadowless $\f$-acyclic $2$-complex over every field $\f$. Recall that assuming Artin's conjecture there are infinitely many shadowless $2$-dimensional $\Q$-almost hypertrees. We saw in Section~\ref{sec:f2} that every $\f_2$-almost hypertree has a shadow and there we discussed its minimal caridnality. We now complement this by seeking the largest number of $2$-faces in shadowless hyperforests. Our construction works at once {\em for all fields} since it is based on the combinatorial property of $2$-collapsibility.
%
%
%
%

\begin{theorem}
\label{thm:cls_acyc}
For every odd integer $n$, there exists a $2$-collapsible $2$-complex $A=A_n$ with ${{n-1}\choose {2}} - (n+1)$ faces that remains $2$-collapsible after the addition of any new face. In particular, this complex is acyclic and shadowless over every field.
\end{theorem}

\begin{proof}
The vertex set $V=V(A)$ is the additive group $\Z_n$. All additions here are done $\bmod ~n$. Edges in $A$ are denoted $(x,x+a)$ with $a<n/2$, and such an edge is said to have {\em length} $a$. Also, for $a>1$, $b=\floor{\frac a2}$ is uniquely defined subject to $1\le b <a < n/2$. For every $x\in \Z_n$ and $n/2>a>1$ we say that the edge $(x, x+a)$ {\em yields} the face $\rho_{x,a}:=\{x,x+a,x+\floor{\frac a2}\}$ of {\em length} $a$. These are $A$'s $2$-faces:

\[\{\rho_{x,a} ~|~ n/2>a\ne 1,3,~x\in \Z_n\}\]

It is easy to $2$-collapse $A$ by collpasing $A$'s faces in decreasing order of their lengths. In each phase of the collapsing process, the longest edges in the remaining complex are exposed and can be collapsed.

It remains to show that the complex $A\cup\{\sigma\}$ is $2$-collapsible for every face $\sigma=\{x,y,z\}\notin A$. To this end, let us carry out as much as we can of the "top-down" collapsing process described above. Clearly some of the steps of this process become impossible due to the addition of $\sigma$, and we now turn to describe the complex that remains after all the possible steps of the previous collapsing process are carried out. Subsequently we show how to $2$-collapse this remaining complex and conclude that $A\cup\{\sigma\}$ is $2$-collpasible, as claimed.

For every $n/2>a\ge 1,~x\in \Z_n$ we define a subcomplex $C_{x,x+a} \subset A$. If $a=1$ or $3$, this is just the edge $(x, x+a)$. For all $n/2>a\ne 1, 3$ it is defined recursively as $C_{x,~x+\floor{\frac a2}}\cup C_{x+\floor{\frac a2},~x+a}\cup \{\rho_{x,a}\}$.

Note that $C_{x,y}$ is a triangulation of the polygon that is made up of the edge $(x,y)$ and $C(x,y)$'s edges of lengths $1$ and $3$.

Our proof will be completed once we (i) Observe that this remaining complex is $\Delta_\sigma:=\{\sigma\}\cup C_{x,y}\cup C_{x,z}\cup C_{y,z}$, and (ii) Show that $\Delta_\sigma$ is $2$-collapsible.

Indeed, toward (i), just follow the original collapsing process and notice that $\Delta_\sigma$ is comprised of exactly those faces in $A$ that are affected by the introduction of $\sigma$ into the complex.

We will show (ii) by proving that the face $\sigma$ can be collapsed out of $\Delta_\sigma$. Consequently, $\Delta_\sigma$ is $2$-collapsible to a subcomplex of the $2$-collapsible complex $A$. 

As we show below

\begin{claim}\label{clm:only_leaf}
There exists a vertex in $\Delta_\sigma$ which belongs to exactly one of the complexes $C_{x,y}, C_{x,z}$ or $C_{y,z}$.
\end{claim}

This allows us to conclude that the face $\sigma$ can be collapsed out of $\Delta_\sigma$. 
Say that the vertex $v$ is in $C_{x,y}$ and only there, and let $e$ be some edge of length $1$ or $3$ in $C_{x,y}$ that contains $v$. Follow the recursive consturction of $C_{x,y}$ as it leads from $(x,y)$ to $e$. Every edge that is encountered there appears only in the polygon $C_{x,y}$. By traversing this sequence in reverse, we collpase $\sigma$ out of $\Delta_\sigma$. 
\end{proof}

\begin{proof}[Proof of Claim \ref{clm:only_leaf}]
By translating $\bmod ~n$ if necessary we may assume that $x=0$ and $0<y,z-y<\frac n2$. If $z>\frac n2$, then $V(C_{0,y})\subseteq\{0,...,y\}$, $V(C_{y,z})\subseteq\{y,...,z\}$ and $V(C_{z,0})\subseteq\{z,...,n-1,0\}$, so their vertex sets are nearly disjoint altogether.

We now consider the case $z<\frac n2$ and assume by contradiction that the claim fails for $\sigma=\{0,y,z\}$. We want to conclude that $\sigma \in A$, and in fact $\sigma \in C_{0,z}$. By the recursive construction of $C_{0,z}$, this, in other words, means that both edges $(0,y)$ and $(y,z)$ are in $C_{0,z}$. We only prove that $(0,y)\in C_{0,z}$, and the claim $(y,z)\in C_{0,z}$ follows by an essentially identical argument.

So we fix $0<y<z<\frac n2$ and we want show that
\begin{equation}\label{eqn:unique_vertex}
\mbox{If $(0,y)\notin C_{0,z}$ then $V(C_{0,z})\cap [0,y] \neq V(C_{0,y})$.}
\end{equation}

Consequently, there is a vertex $v$ in $[0,y]$ which belongs to exactly one of the complexes $C_{0,z}$ and $C_{0,y}$. If such a $v<y$ exists, we are done, since $C_{y,z}$ has no vertices in $[0,y-1]$. Otherwise, 
\[
V(C_{0,z})\cap[0,y-1] = V(C_{0,y})\setminus \{y\}~~~\mbox{and}~~~y\notin C_{0,z}.
\]
But the vertices of $C_{0,z}$ form an increasing sequence from $0$ to $z$ with differences $1$ or $3$, so either $(y-2,y+1)\in C_{0,z}$ or $(y-1,y+2)\in C_{0,z}$. In the former case, both $y-2$ and $y$ are vertices in $C_{0,y}$, and therefore $y-1 \in C_{0,y}$ and consequently $y-1 \in C_{0,z}$, contrary to the assumption that the edge $(y-2,y+1)$ is in $C_{0,z}$. In the latter case, $y+2 \in C_{0,z}$ and $y+1 \notin C_{0,z}$. Which vertex succeeds $y$ in $C_{y,z}$? If $(y,y+1)\in C_{y,z}$ then $y+1$ belongs only to $C_{y,z}$. If $(y,y+3)\in C_{y,z}$ then $y+2$ is only in $C_{0,z}$.

We prove the implication~(\ref{eqn:unique_vertex}) by induction on $y$. The base cases where $y=1$ or $y=3$, are straightforward. If $(0,\floor{\frac y2})\notin C_{0,z}$ then by induction $V(C_{0,z})\cap [0,\floor{\frac y2}] \neq V(C_{0,\floor{\frac y2}})$. But $V(C_{0,y})\cap [0,\floor{\frac y2}] = V(C_{0,\floor{\frac y2}})$ and the conclusion that $V(C_{0,z})\cap [0,y] \neq V(C_{0,y})$ follows. We now consider what happens if $(0,\floor{\frac y2})\in C_{0,z}$. Which edge has yielded the $2$-face of $C_{0,z}$ that contains the edge $(0,\floor{\frac y2})$? It can be either $(0, 2\cdot\floor{\frac y2})$ or $(0, 2\cdot\floor{\frac y2}+1)$. But one of these two edges is $(0,y)$ which, by assumption, is not in $C_{0,z}$, so it must be the other one. Namely, either $y=2r$ and $(0,2r+1)\in C_{0,z}$ or $y=2r+1$ and $(0,2r)\in C_{0,z}$.

Let us deal first with the case $y=2r$. 
Assume, in contradiction to ~(\ref{eqn:unique_vertex}), that $V(C_{0,z})\cap [0,2r] = V(C_{0,2r})$. In particular $V(C_{0,z})\cap [r,2r] = V(C_{0,2r})\cap [r,2r]$. But since $(0,2r+1)$ is an edge of $C_{0,z}$ it also follows that $V(C_{0,2r+1})\cap [r,2r] = V(C_{0,z})\cap [r,2r]$. Therefore,
\begin{equation*}\label{eqn:vertex_segment_equality}
V(C_{0,2r+1})\cap [r,2r] = V(C_{0,2r})\cap [r,2r].
\end{equation*}
 
By the recursive consturction of $C_{0,2r+1}$ and $C_{0,2r}$ we obtain that $
V(C_{r,2r+1})\cap [r,2r] = V(C_{r,2r})$. By using the rotational symmetry of $A$ we can translate this equation by $r$ to conclude that
$V(C_{0,r+1})\cap [0,r] = V(C_{0,r})$. By induction, using the contrapositive of Equation~(\ref{eqn:unique_vertex}) this implies that $(0,r)\in C_{0,r+1}$ hence $r=1$. However, $C_{0,z}$ cannot contain both $(0,3)$ and $(0,1)$ so we are done.

The argument for $y=2r+1$ is essentially the same and is omitted.
\end{proof}

\section{Open Problems}
\label{section:open}
\begin{itemize}
\item
There are several problems that we solved here for $2$-dimensional complexes. It is clear that some completely new ideas will be required in order to answer these questions in higher dimensions. In particular it would be interesting to extend the construction based on arithmetic triples for $d>2$.
\item
An interesting aspect of the present work is that the behavior over $\f_2$ and $\Q$ differ, some times in a substantial way. It would be of interest to investigate the situation over other coefficient rings.
\item
How large can an acyclic closed set over $\f_2$ be?
Theorem \ref{thm:cls_acyc} gives a bound, but we do not know the exact answer yet.
\item
We still do not even know how large a $d$-cycle can be. In particular, for which integers $n, d$ and a field $\f$ does there exist a set of ${n-1 \choose d}+1$ $d$-faces on $n$ vertices such that removing any face yields a $d$-hypertree over $\f$?
\item
Many basic (approximate) enumeration problems remain wide open.
How many $n$-vertex $d$-hypertrees are there? What about $d$-collapsible complexes? A fundamental work of Kalai~\cite{kalai} provides some estimates for the former problem, but these bounds are not sharp. In one dimension there are exactly $\frac{(n-1)!}{2}$ inclusion-minimal $n$-vertex cycles. We know very little about the higher-dimensional counterparts of this fact.
\end{itemize}

\appendix
\section{Appendix - Proof of Theorem \ref{thm:5}}
Let $h(\gamma, x,y) = g(\gamma,x,y) - \frac 34 = \gamma - x^2 - \frac
34 - y(\gamma - \frac 12 - x)$. We need to show that $h \geq 0$
under the conditions of the optimization problem. This involves some case analysis.

First note $\gamma-\frac 12 - x \geq 0$ by condition 3, so that for fixed $\gamma, x$
we have that $h$ is a decreasing function of $y$. Thus, to minimize $h$, we need to determine the largest possible value of $y$.

\begin{enumerate}
\item We first consider the range $\gamma \leq 2$. Here condition 4 is redundant, and $y \leq \min \{x, \gamma - \frac 12 - x \}$.
 \begin{enumerate}
 \item We further restrict to the range $x \leq \frac{\gamma}{2} - \frac 14$, where
$x \leq \gamma - \frac 12 - x$, so the largest feasible value of $y$ is $y=x$. Note that $h|_{y=x} = \gamma - \frac 34 - x(\gamma - \frac 12)$. But $\gamma - \frac 12 \geq 0$ by condition 1, so $h$ is minimized by maximizing $x$, namely taking $x=  \frac{\gamma}{2} - \frac 14$. This yields $h=\frac 18-\frac 12(\gamma-1)(\gamma-2)$ which is positive in the relevant range $2\ge \gamma \ge 1$.
 \item In the complementary range $\frac{\gamma}{2} - \frac 14 \leq x$ the largest value for $y$ is $y = \gamma - \frac 12 - x$ which yields $h =  \gamma - x^2 -\frac 34 - (\gamma - \frac 12 - x)^2 =-2(x-\frac{\gamma}{2})(x-\frac{\gamma-1}{2})-\frac{(\gamma-1)(\gamma-2)}{2}$.
It suffices to check that $h\ge 0$ at both extreme value of $x$, namely $\frac{\gamma}{2} - \frac 14$ and $\gamma/2$. Also $h =0$ only at $x= \gamma/2$ with $\gamma = 1$ or $2$.
  \end{enumerate}

\item In the complementary range $\gamma \geq 2$, condition $3$
  is redundant and condition 4 takes over.
  \begin{enumerate}
  \item
Assume first that $x \leq \frac{1 + \gamma}{4}$, then $x \leq \frac{1+\gamma}{2} - x$ and the extreme value for $y$ is $y = x$. Again $h|_{y=x} = \gamma - \frac 34 - x(\gamma - \frac 12)$ and now the largest possible value of $x$ is $x = \frac{1 + \gamma}{4}$ which yields $h=\frac{(5-2\gamma)(\gamma - 1)}{8}$. This is positive at the range $\frac 94\ge \gamma\ge 2$.

 \item
When $x \geq\frac{1 + \gamma}{4}$ the minimum $h$ is attained at $y = \frac{1+\gamma}{2} - x$, so that
$$h = \gamma - x^2 - \frac 34 - (\frac{1+\gamma}{2} - x)\cdot (\gamma - \frac
12 - x)=-2(x-1)(x+1-\frac{3\gamma}{4})-\frac 12(\gamma-2)(\gamma-\frac 52).$$

For fixed $\gamma$ it suffices to check that $h\ge 0$ at the two ends of the range $1\ge x\ge\frac{1+\gamma}{4}$. At $x=1$ we get $h= -\frac 12(\gamma-2)(\gamma-\frac 52)$ which is nonnegative when $\frac 94\ge\gamma\ge 2$ with $h=0$ only at $\gamma=2$. When $x=\frac{1+\gamma}{4}$,  we get $h=\frac{(5-2\gamma)(\gamma - 1)}{8}$ which is positive for $\frac 94\ge\gamma\ge 2$.
 \end{enumerate}

\end{enumerate}

To sum up, $h\ge 0$ throughout the relevant range with two point where $h=0$, namely $\gamma=2, x=1, y=\frac 12$ and $\gamma=1, x=\frac 12, y=0$.


\begin{thebibliography}{10}

\bibitem{adin}
Ron~M. Adin.
\newblock Counting colorful multi-dimensional trees.
\newblock {\em Combinatorica}, 12(3):247--260, 1992.

\bibitem{bab_kah}
Eric Babson, Christopher Hoffman, and Matthew Kahle.
\newblock The fundamental group of random 2-complexes.
\newblock {\em Journal of the American Mathematical Society}, 24(1):1--28,
  2011.

\bibitem{BK}
Anders Bj{\"o}rner and Gil Kalai.
\newblock An extended euler-poincar{\'e} theorem.
\newblock {\em Acta Mathematica}, 161(1):279--303, 1988.

\bibitem{de_caen}
D.~de Caen, D.L. Kreher, and J.~Wiseman.
\newblock { On constructive upper bounds for the Tur\'an numbers T(n,2r+1,r)}.
\newblock {\em Congressus Numerantium}, 65:277--280, 1988.

\bibitem{farber}
Daniel Cohen, Armindo Costa, Michael Farber, and Thomas Kappeler.
\newblock Topology of random 2-complexes.
\newblock {\em Discrete Computational Geometry}, 47(1):117--149, 2012.

\bibitem{DNRR}
Rajendraprasad Deepak, Ilan Newman, Rogers Matthews, and Yuri Rabinovich.
\newblock Extremal problems on cycles in simplicial complexes.
\newblock {\em in preparation}.

\bibitem{duval}
Art Duval, Caroline Klivans, and Jeremy Martin.
\newblock Simplicial matrix-tree theorems.
\newblock {\em Transactions of the American Mathematical Society},
  361(11):6073--6114, 2009.

\bibitem{gromov}
Mikhail Gromov.
\newblock Singularities, expanders and topology of maps. part 2: From
  combinatorics to topology via algebraic isoperimetry.
\newblock {\em Geometric and Functional Analysis}, 20(2):416--526, 2010.

\bibitem{kalai}
Gil Kalai.
\newblock Enumeration of $\mathbb{Q}$ -acyclic simplicial complexes.
\newblock {\em Israel Journal of Mathematics}, 45(4):337--351, 1983.

\bibitem{keevash_survey}
Peter Keevash.
\newblock Hypergraph tur{\'a}n problems.
\newblock {\em Surveys in combinatorics}, 392:83--140, 2011.

\bibitem{lin_mesh}
Nathan Linial and Roy Meshulam.
\newblock Homological connectivity of random 2-complexes.
\newblock {\em Combinatorica}, 26(4):475--487, 2006.

\bibitem{sum_complex}
Nathan Linial, Roy Meshulam, and Mishael Rosenthal.
\newblock Sum complexes - a new family of hypertrees.
\newblock {\em Discrete \& Computational Geometry}, 44(3):622--636, 2010.

\bibitem{lubotzky}
Alexander Lubotzky, Beth Samuels, and Uzi Vishne.
\newblock Ramanujan complexes of type $\tilde {A}_d$.
\newblock {\em Israel Journal of Mathematics}, 149(1):267--299, 2005.

\bibitem{artin}
Pieter Moree.
\newblock Artin's primitive root conjecture--a survey.
\newblock 2012.

\bibitem{V-con}
Ilan Newman and Yuri Rabinovich.
\newblock On multiplicative $\lambda$-approximations and some geometric
  applications.
\newblock {\em SIAM Journal on Computing}, 42(3):855--883, 2013.

\bibitem{yuval_msc}
Yuval Peled.
\newblock Combinatorics of simplicial cocycles and local distributions in
  graphs.
\newblock Master's thesis, Hebrew University of Jerusalem, 2012.

\bibitem{sidorenko}
A.~Sidorenko.
\newblock The method of quadratic forms and tur\'an's combinatorial problem.
\newblock {\em Moscow University Mathematics Bulletin}, 37(1):1--5, 1982.

\end{thebibliography}
\end{document}